\documentclass[reqno, 11pt]{amsart}
\usepackage{amssymb}
\usepackage[nesting]{hyperref}

\usepackage[pdftex]{graphicx}
\usepackage{listings}
\usepackage{multirow}
\usepackage{placeins}
\usepackage{color}
\usepackage{subfigure}
\usepackage{lscape}
\usepackage{dsfont}
\usepackage{tikz}

\newcommand{\BlackBoxes}{\global\overfullrule5pt}

\BlackBoxes

\newcommand{\R}{\mathbb{R}}
\newcommand{\N}{\mathbb{N}}

\newcommand{\Eop}{\mathbb{E}}
\newcommand{\Prob}{\mathbb{P}}
\newcommand{\F}{\mathcal{F}}

\newcommand{\Hist}{{\mathcal H}}

\newtheorem{theorem}{Theorem}

\newtheorem{lemma}[theorem]{Lemma}

\theoremstyle{definition}

\newtheorem{remark}[theorem]{Remark}
\newtheorem{definition}[theorem]{Definition}
\numberwithin{equation}{section} \numberwithin{theorem}{section}

\def\0{\kern0pt\-\nobreak\hskip0pt\relax}

\makeatletter
\AtBeginDocument{%
 \def\@serieslogo{%
 \vbox to\headheight{%
 \parindent\z@ \fontsize{6}{7\p@}\selectfont
 \vss}}}

\def\makeoverbar#1#2#3#4#5#6#7{%
 \setbox0=\hbox{$\m@th#2\mkern#5mu{{}#3{}}\mkern#6mu$}%
 \setbox1=\null \dimen@=#4\fontdimen8#13 \dimen@=3.5\dimen@
 \advance\dimen@ by \ht0 \dimen@=-#7\dimen@ \advance\dimen@ by \wd0
 \ht1=\ht0 \dp1=\dp0 \wd1=\dimen@
 \dimen@=\fontdimen8#13 \fontdimen8#13=#4\fontdimen8#13
 \rlap{\hbox to \wd0{$\m@th\hss#2{\overline{\box1}}\mkern#5mu$}}
 \fontdimen8#13=\dimen@}

\def\mylabel#1#2{{\def\@currentlabel{#2}\label{#1}}}

\makeatother

\begin{document}


\makeatletter \providecommand\@dotsep{5} \makeatother

\title[Optimal Control of Partially Observable PDMP]{Optimal Control of Partially Observable Piecewise Deterministic Markov Processes}

\author[N. \smash{B\"auerle}]{Nicole B\"auerle${}^*$}
\address[N. B\"auerle]{Department of Mathematics,
Karlsruhe Institute of Technology, D-76128 Karlsruhe, Germany}

\email{\href{mailto:nicole.baeuerle@kit.edu}
{nicole.baeuerle@kit.edu}}

\author[D. \smash{Lange}]{Dirk Lange${}^\dag$}
\address[D. Lange]{Department of Mathematics,
Karlsruhe Institute of Technology, D-76128 Karlsruhe, Germany}

\email{\href{mailto:dirk.lange@ukit.edu} {dirk.lange@kit.edu}}


\begin{abstract}
In  this paper we consider a control problem for a Partially Observable Piecewise Deterministic Markov Process of the following type: After the jump of the process the controller receives a noisy signal about the state and the aim is to control the process continuously in time in such a way that the expected discounted cost of the system is minimized. We solve this optimization problem by reducing it to a discrete-time Markov Decision Process. This includes the derivation of a filter for the unobservable state. Imposing sufficient continuity and compactness assumptions we are able to prove the existence of optimal policies and show that the value function satisfies a fixed point equation. A generic application is given to illustrate the results.
\end{abstract}
\maketitle

\vspace{0.5cm}
\begin{minipage}{14cm}
{\small
\begin{description}
\item[\rm \textsc{ Key words}  ]
{\small Partially Observable Piecewise Deterministic Markov Process, Markov Decision Process, Filter, Updating-Operator}
\item[\rm \textsc{AMS subject classifications}]{\small Primary 60J25, 90C40 Secondary 93E11}
\end{description}
}
\end{minipage}

\section{Introduction}\label{sec:intro}
{\em Piecewise Deterministic Markov Processes} (PDMP) are characterized by three local characteristics: The drift, describing the deterministic movement between two jumps of the process, the jump intensity, governing the density of the probability distribution of the inter-jump times as well as the jump transition kernel, the probability distribution on the set of possible post-jump states given the current state of the process right before the jump. A PDMP thus starts in an initial state to then follow the deterministic path defined by the drift up to the first jump time.

Classical optimization problems can be formulated for PDMPs such as reward maximization or cost minimization. Minimum  expected average cost problems (see, e.g., \cite{BaeuerleDiscounted2001}, \cite{CostaDufcontav} or \cite{CostaDuf10}) as well as  minimum  expected total discounted  cost problems (see e.g. \cite{alm01}, \cite{DY}, \cite{ForSchal2004}) have intensively been treated for PDMP control problems. Optimal policies are in general relaxed controls, i.e. a control action is a probability distribution on the action space. The idea of reducing the continuous time control problem of a PDMP to a discrete time Markov Decision Process (MDP) is due to Yushkevich, see \cite{Yushkev80}. Actually, as the movement of the process between two jumps is deterministic, a pure post-jump consideration is sufficient for the treatment of optimal control problems for PDMPs.

The range of possible applications of the general PDMP control theory is broad. There are applications in insurance \cite{Schael1998}, communication networks \cite{Chafai2010}, reliability \cite{DD02}, neurosciences \cite{Pakdaman2010} and biochemics \cite{Lygeros2008} to only list a very short overview that illustrates the huge variety of domains of application.

In terms of pure mathematical treatment of PDMP control problems, the status up to 1993 can be found in \cite{Davis93}. Since then, important steps in the further development of this theory were, amongst others: In \cite{CostaRay2000} the authors consider impulse control of PDMPs without continuity or differentiability assumptions on the state. In \cite{alm01}, the control problem in continuous time is reduced to a problem in discrete time while working under even lower regularity assumptions. General conditions such as semi-analytic value functions or universally measurable selectors are applied. \cite{ForSchal2004} then considers, in contrast to the earlier works, problems with only locally bounded running cost functions. They show absolute continuity for the value function and that the value function is a (weak) solution of the Hamilton-Jacobi-Bellmann equation. In addition, they derive sufficient conditions for the existence of optimal deterministic feedback controls.

Later, with \cite{sdg10} and \cite{BrandSapDuf13} new results on numerical methods for optimal stopping problems for PDMPs appeared. In both works, the embedded process of the underlying PDMP is discretized by quantization. Remarkable about the paper \cite{BrandSapDuf13} is, however, that they treat an optimal stopping problem for a PDMP under partial observation. Such a setting is also considered in \cite{MJ03} where a replacement problem under partial information is considered. Whereas in  \cite{BrandSapDuf13} new information is only received after a jump, the information in \cite{MJ03} is received via monitoring at equidistant inspection times.  Besides these papers there are only very few works treating PDMP control problems under partial observation.
In \cite{KirchRunggaldier05}, a special convex hedging problem on a financial market with price processes following a geometric Poisson-distribution is considered. In the second part of this work, partial observation is modeled by assuming an unknown jump intensity. In \cite{bl10}, a problem of optimal inventory management is considered. Here, partial observation is modeled by assuming censored observations.

General works on PDMP control problems under partial observation do not exist yet. For their stopping problem, the authors of \cite{BrandSapDuf13} suggest to model partial observation by assuming only noisy measurement of the post-jump state of the PDMP which for other times than jump times, is assumed completely unobservable. Stopping, however is a very special control problem with only two control actions: {stop} or {continue}.

In this paper the first aim is to define a general model of a controlled PDMP under partial observation with the discounted cost criterion. We assume as in \cite{BrandSapDuf13} that the controller receives a noisy measurement of the post-jump state of the PDMP. Then we show how this continuous-time control problem can be reduced to a classical discrete-time MDP with a state space consisting of probability measures. This involves the derivation of a filter for the unobservable state.  We  next impose some continuity and compactness assumptions along with the introduction of a regularized filter in order to guarantee the existence of optimal policies. A problem which is known to be notoriously difficult (see e.g. \cite{Feinberg1}). Finally the value function of the optimization problem is shown to be a fixed point of an operator and the minimizer of the value function defines a stationary  optimal  policy.

Our paper is organized as follows: In the next section we briefly introduce the notation of an uncontrolled PDMP with partial observation. In Section \ref{sec:cPartialObs} we add controls. The optimization problem itself is explained in  Section \ref{sec:InitialProblem}. In the same section we show how to reduce the problem to a Partially Observable Markov Decision Process and derive the corresponding filter.  Afterwards in Section \ref{sec:oe} we present the optimality equation for the value function and prove existence of optimal controls under our assumptions. A generic application is given in Section \ref{sec:applic}.


\section{An uncontrolled PDMP with partial observation}\label{sec:uncontrolled}
In \cite{Davis84} the class of PDMPs has been introduced as a {\em general class of non-diffusion stochastic models}. A definition of a PDMP based on its infinitesimal generator is given there and thus strongly emphasizing the fact that a PDMP is a priori a continuous-time process. Recent publications such as \cite{BrandSapDuf13} or \cite{ForSchal2004} introduce a PDMP following an axiomatic approach stating a set of properties of a PDMP. We will follow the latter approach in this paper.

We first define an uncontrolled PDMP with partial observation before we consider in Section  \ref{sec:cPartialObs} controlled PDMPs under partial observation. An informal description of a PDMP with partial observation  is as follows: The process $(Y_t)_{t \geq 0}$ with values in $\R^d$ first evolves in a  deterministic way according to a certain drift $\Phi$. The \emph{drift} $\Phi: \R^d \times \R_+ \rightarrow \R^d$ is continuous and the mapping $t \mapsto \Phi(\cdot, t)$ is a semi-group with respect to concatenation of mappings, i.e. for all $y \in \R^d$ and $s,t>0$:
\begin{equation}
\Phi(y, t+s) =  \Phi(\Phi(y,s), t).
\end{equation}
$\Phi(y,t)$ is the state of the process $t$ time units after the last jump when the state directly after the jump was $y$.
Often in applications the drift $\Phi$ is given by a differential equation
\begin{equation}\label{eq:PDEDefinedDrift}
\frac{d}{dt} \Phi(y,t) = b\big(\Phi(y,t)\big), \hspace{0.5cm} \Phi(y,0)=y
\end{equation}
where  $b: \R^d \rightarrow \R^d$ is a vector field guaranteeing for all $y \in \R^d$ a unique componentwise continuous solution.

At the random time $T_1$  the process jumps unpredictably to a new state where the deterministic evolution continues until the next jump occurs. The  jump times   $0:=T_0< T_1<  \dots$ are $\R_+$-valued random variables such that $S_n := T_n-T_{n-1}, n\in\N$, $S_0 :=0$ and $T_n<T_{n+1}$ if $T_n<\infty$ else $T_n=T_{n+1}$. The jump times are generated by a \emph{jump rate} or \emph{intensity} $\lambda: \R^d \rightarrow (0, \infty)$ which is a measurable mapping of the state. A \emph{transition kernel}  $Q$ from $\R^d$ to $\R^d$  describes the probability $Q(B|y)$ that the process jumps into set $B$ given the state before the jump is $y$.

We assume now that the state of the PDMP cannot be observed directly. Several models might arise from this imperfect information about the system state. In view of applications to problems from telecommunications, engineering, supply chain or finance, the idea is to assume that one can at least measure (or estimate) the true state of the system with some measurement noise. We assume that at jump times of the PDMP we receive new information about the state. More precisely let $(\epsilon_n)_{n \in \N}$ be a sequence of $\R^d$-valued independent and identically distributed random variables $\epsilon_n: \Omega \to \R^d$ that are independent from all other random variables. We call $\epsilon_n$ \emph{observation noise} and denote its distribution by $Q_\epsilon$. We assume that the agent is able to observe $X_n := {Y}_{T_n}+\epsilon_n$ directly after the jump at time $T_n$.

Given the data $(\Phi,\lambda,Q,Q_\epsilon)$, an initial state $y$ and its observation $x$ there exists a probability space $(\Omega, \F,\Prob_{x,y})$  carrying the random variables $(T_n)$, $(Y_{T_n})$ and $(\epsilon_n)$ such that  $\Prob_{x,y}(Y_0=y, X_0=x)=1$ and for all $n\in\N, t\ge 0, C, D\in \mathcal{B}_d$, where $\mathcal{B}_d$ is the $\sigma$-algebra of Borel sets in $\R^d$, it holds that
\begin{eqnarray}
\nonumber &&\Prob_{x,y}(S_n \leq t,   {Y}_{T_n} \in C, X_n \in D  \mid S_0,  {Y}_{T_0}, X_0, \dots, S_{n-1},  {Y}_{T_{n-1}}, X_{n-1})\\
\nonumber &=& \Prob_{x,y}(S_n \leq t,   {Y}_{T_n} \in C, X_n \in D  \mid  {Y}_{T_{n-1}})\\
\nonumber &=&\int_{0}^t\int_C \Prob_{y}(X_n \in D \mid  {Y}_{T_n}=y') \\
\nonumber && \hspace*{3cm}\Prob_{x,y}(ds,dy' \mid  {Y}_{T_{n-1}}) \\
\nonumber &=&\int_0^t\int_C Q_\epsilon(D-y')  \exp\big(-\Lambda({Y}_{T_{n-1}},s)\big) \lambda\big(\Phi({Y}_{T_{n-1}},s)\big)  Q\big(dy'|\Phi({Y}_{T_{n-1}},s)\big) ds\\
&=&  \int_0^t \exp\big(-\Lambda({Y}_{T_{n-1}},s)\big) \lambda\big(\Phi({Y}_{T_{n-1}},s)\big)  \int_C Q_\epsilon(D-y')   Q\big(dy'|\Phi({Y}_{T_{n-1}},s)\big) ds
\end{eqnarray}
where $\Lambda(y,t) := \int_0^t \lambda\big(\Phi(y,s)\big)ds$.
The (unobservable) process itself is then given by
\begin{equation}
Y_t  := \Phi(Y_{T_n},t-T_n), \quad  \mbox{ for } T_n\le t<T_{n+1}, n\in\N_0.
\end{equation}
In what follows  we define the embedded process of $(Y_t)$ by $\hat{Y}_n := Y_{T_n}$ in order to ease notation. Note that $(T_n,\hat{Y}_n,X_n)$ is a marked point process. We call such a process {\it Partially Observable Piecewise Deterministic Markov Process (POPDMP)}. In the general definition of a PDMP boundary points of the state space may exist which force jumps back into the interior of the state space when reached. In order to ease the following analysis we neglect such a behavior in our model. It would have a severe impact on the filter which we need later.

\section{Controlled POPDMP under Partial Observation}\label{sec:cPartialObs}
Now we assume that the POPDMP can be controlled in continuous time. The set of actions is denoted by $A$. In order to  prove existence of optimal policies later we need the following assumption. \\

{\em Assumption:}
\begin{description}
  \item[(C1)] The action space $A$ is a compact metric space.
\end{description}

\vspace*{0.4cm}

We denote by  $\mathcal{P}(A)$ the set of all probability measures on $(A,\mathcal{B}_A)$ with the weak topology. From the theory of deterministic control it is well-known that in order to prove the existence of optimal controls we have to work with relaxed controls. The space $\mathcal{R}$ of \emph{relaxed controls} is given by $$ \mathcal{R} := \left\{ r  : [0,\infty) \rightarrow \mathcal{P}(A) \; \mid r \; \mathrm{ is \; measurable} \right\}.$$
On $\mathcal{R}$ we work with the Young topology (for convergence in Young topology see the appendix). Note that under assumption (C1), the space  $\mathcal{R}$ is compact under the Young topology (see e.g. \cite{Davis93} Proposition 43.3 and Definition 43.4 together with the comment thereafter).

Next we define the set of observable histories up to time $T_n$. Let $\Hist_0 :=  \R^d$ and for $n\in\N$
$$ \Hist_n := \Hist_{n-1} \times\mathcal{R} \times \R_+ \times \R^d$$
and endow this space with the corresponding product $\sigma$-algebra. An element denoted by \linebreak $h_n=(x_0,r_0,s_1,x_1, \dots, r_{n-1},s_n,x_n) \in \Hist_n$ is called \emph{observed history up to time $T_n$}. It consists of the received signals, the chosen controls and the inter-arrival times of jumps up to $T_n$. A decision rule for the period $[T_n,T_{n+1})$ is a measurable mapping $$\pi_n^P: \Hist_n \times [0, \infty) \to \mathcal{P}(A).$$
The upper $P$ in the notation stands for {\em piecewise}.
For $n \in \N_0$, the space of all decision rules for the period $[T_n, T_{n+1})$ is denoted by $\Pi_n^P$ and the space of all history dependent relaxed piecewise open loop policies is defined as
$$\Pi^P := \Pi_0^P\times \Pi_1^P\times \ldots.$$
Executing a history dependent relaxed piecewise open loop policy $\pi^P=(\pi_0^P, \pi_1^P,\dots) \in \Pi^P$ means executing, at time $t \geq 0$
\begin{equation}\label{eq:DarstellungPit}
\pi_t:= \sum_{n=0}^\infty 1_{\{T_n \leq t < T_{n+1}\}}(t) \cdot \pi_n^P(H_n, t-T_n),
\end{equation}
where $H_n= (X_0,\pi_0^P(X_0,\cdot),S_1,X_1,\ldots,\pi^P_{n-1}(H_{n-1},\cdot),S_n,X_n)$. There is an alternative way of introducing policies which will be crucial later on and which we explain now. A \emph{discrete time history dependent relaxed control policy} is a sequence $\pi^D := ( \pi_0^D, \pi_1^D, \dots)$ of discrete time history dependent decision rules where $\pi_n^D: \Hist_n \to \mathcal{R}$ is measurable. The upper $D$ in the notation stands for {\em discrete}.  Note that $\pi_n^D(h_n)$ is a function in time and $\pi_n^D(h_n)(t)$ is the (randomized) action applied $t$ time units after the $n$-th jump at time $T_n$. Here instead of a continuous-time control we have a discrete-time policy which is applied after jump time points and which now consists of functions.
We write $\Pi_n^D$ for the set of all discrete time history dependent decision rules at stage $n$ and define the set of all discrete time history dependent relaxed control policies as $\Pi^D:= \Pi_0^D\times \Pi_1^D\times\ldots$. Note that the following statement holds which is essentially a measurability issue. For a proof see \cite{lange17} Theorem 2.11.

\begin{lemma}[Correspondence Lemma] \label{theo:CorrespondenceTheorem} Let $n \in \N_0$. For every $\pi_n^P \in \Pi_n^P$ there exists  $\pi_n^D \in \Pi_n^D$ such that
\begin{equation}
\pi_n^P(h_n, t) = \pi_n^D(h_n)(t) \hspace{0.5cm}  \mathrm{a.e. ~on~} \R_+~ ~\mbox{ for all } h_n \in \Hist_n
\end{equation}
and vice-versa.
\end{lemma}

Upon choosing a policy in $\Pi^P$ we are able to control the data of our POPDMP in the following way. Suppose the history $h_n$ is given up to time $T_n$ and $ \pi_n^D(h_n)=r$. Then on the time interval $[T_n,T_{n+1})$ the relaxed control $r$ influences the drift which we denote in general by $\Phi^r$ and  $\Phi^r: \R^d \times \R_+ \rightarrow \R^d$ is continuous and the mapping $t \mapsto \Phi^r(\cdot, t)$ is a semi-group with respect to concatenation of mappings, i.e. for all $y \in \R^d$ and $s,t>0$:
\begin{equation*}
\Phi^r(y, t+s) =  \Phi^r(\Phi^r(y,s), t).
\end{equation*}
For example let $b: \R^d \times A \rightarrow \R^d$ be a vector field such that for all $y \in \R^d$ and  all relaxed controls $r \in \mathcal{R}$ the initial value problem
\begin{equation}\label{eq:PDEControlledDrift}
\frac{d}{dt} \Phi^{r}(y,t) = \int_A b(\Phi^{r}(y,t),a) ~r_t(da), \hspace{0.5cm} \Phi^{r}(y,0)=y
\end{equation}
has a unique componentwise continuous solution $\Phi^{r}(y, \cdot) : [0, \infty) \rightarrow \R^d$. Then $\Phi^r$ could be such a drift function. The relaxed control also influences the measurable \emph{jump rate} $\lambda^A: \R^d\times A \rightarrow (0, \infty)$ and the action which is applied at the time point of a jump influences the transition kernel $Q^A$ from $\R^d\times A$ to $\R^d$.

\begin{definition} [Controlled POPDMP] \label{def:ControlledPDMP} A \emph{Controlled Partially Observable Piecewise Deterministic Markov Process}  with \emph{local characteristics} $(\Phi^r, \lambda^A, Q^A, Q_\epsilon)$ is a stochastic process $(Y_t)_{t \geq 0}$ that satisfies the following properties:
Fix $\pi \in \Pi^P$ (we write $\pi$ here instead of $\pi^P$ to ease notation) and an initial state $y$ with observation $x$. There exists a probability space $(\Omega, \F,\Prob_{x,y}^{\pi})$ which carries  random variables $(T_n)$, $(\hat{Y}_n)$, $(\epsilon_n)$ such that $\Prob_{x,y}^{\pi}(Y_0=y,X_0=x)=1$ and for all $t \geq 0, n \in \N_0$ and $C,D\in \mathcal{B}_d$ it holds that:
\begin{eqnarray}
\nonumber&&\Prob_{x,y}^{\pi}(S_n \le  t, \hat{Y}_n\in C, X_n \in D |S_0, \hat{Y}_0, X_0, \pi_0 \dots , S_{n-1}, \hat{Y}_{n-1}, X_{n-1}, \pi_{n-1}) \\
&=& \Prob_{x,y}^{\pi}(S_n \le  t, \hat{Y}_n\in C, X_n \in D | \hat{Y}_{n-1}, \pi_{n-1}(H_{n-1})) \\
\nonumber&=& \int_0^t \exp\big( -\Lambda^{\pi_{n-1}}(\hat{Y}_{n-1},s)\big)  \int_A\lambda^A\big(\Phi^{\pi_{n-1}}(\hat{Y}_{n-1},s),a\big)  \\
&& \hspace*{2cm}\int_C Q_\epsilon(D-y') Q^A\big(dy' | \Phi^{\pi_{n-1}}(\hat{Y}_{n-1},s),a\big) \pi_{n-1}(H_{n-1},s)(da)ds
\end{eqnarray}
where $\Lambda^{r}(y,t) := \int_0^t \int_A\lambda^A(\Phi^{r}(y,s),a) r_s(da) ~ds$ and we use the short-hand notation $\Phi^{\pi_{n-1}}$ instead of $\Phi^{\pi_{n-1}(H_{n-1},\cdot)}$. Note that we apply the Correspondence Lemma \ref{theo:CorrespondenceTheorem} here.
The process $(Y_t)$ is then defined by
\begin{equation}
Y_t  := \Phi^{\pi_n}(Y_{T_n}, t- T_n) \mbox{ for } T_n \leq t < T_{n+1}, n \in \N_0.
\end{equation}

\end{definition}

For our existence result we need the following continuity assumptions:\\

{\em Assumption:}
\begin{description}
  \item[(C2)] $\lambda^A: \R^d\times A \rightarrow (0, \infty)$ is continuous and bounded from above by $\bar{\lambda}$ and from below by $\underline{\lambda}>0$.
  \item[(C3)] $Q^A$ is weakly continuous, i.e. $(x,a)\mapsto \int v(z) Q^A(dz|x,a)$ is continuous and bounded for all $v:\R^d\to \R$ continuous and bounded.
\end{description}

Note that (C2) implies that $T_n \uparrow \infty$ $\Prob_{xy}-a.s.$ for all $x,y\in\R^d$.
\vspace*{0.4cm}

\section{The optimization problem}\label{sec:InitialProblem}
In this section we will introduce our optimization problem and transform it into a Markov Decision Process (MDP) which can be solved with standard techniques. We will do this in two steps: First we rewrite our continuous-time control problem for the Partially Observable Piecewise Deterministic Markov Process as a discrete-time control problem for a Partially Observable Markov Decision Process. Then we reduce this Partially Observable Markov Decision Process to a Markov Decision Process with complete observation. This problem will then be solved in the next section.

Let $\beta \in \R_+$ be a discount rate and $c: \R^d \times A \rightarrow \R_+$ be  a measurable cost rate. The initial distribution of $Y_0$ given the observation $X_0=x$ is given by the transition kernel $Q_0(\cdot|x)$. We define the \emph{cost of policy} $\pi\in\Pi^P$ under an initial observation $x \in \R^d$ by (we write $\pi$ instead of $\pi^P$ in order to ease notation)
\begin{equation}\label{eq:CostOfStrategy}
J(x,  \pi) := \int \Eop_{x,y}^{\pi} \left[ \int_0^{\infty} e^{-\beta t} \int_A c(Y_t, a) ~\pi_t(da) ~dt \right] Q_0(dy|x).
\end{equation}

The \emph{value function} of the control model gives the minimal cost under an initial observation $x \in \R^d$ and is defined as
\begin{equation}\label{eq:ValueFunctionControlModelTimeContinuous}
J(x) := \inf_{\pi \in \Pi^P} J(x, \pi) \hspace{0.5cm} \mbox{ for all } ~x \in \R^d.
\end{equation}
The \emph{optimization problem} is then to find, for $x \in \R^d$, a policy $\pi^\star \in \Pi^P$ such that  we get
\begin{equation} J(x) = J(x,\pi^\star). \end{equation}
This problem can be rewritten as a Partially Observable Markov Decision Process in discrete time (POMDP) where we focus on the jump time points only.

\begin{definition}\label{def:translawMDP}
Consider the following discrete-time Partially Observable Markov Decision Model:
\begin{enumerate}
  \item[(i)] The state space of this process is given by $\R_+\times  \R^{2d}$ and a typical state is denoted by $(s,y,x)$. The interpretation of the state is that  $y$ is the (unobservable) state directly after the jump which occured $s$ time units after the previous jump and $x$ is the observation.
  \item[(ii)] The action space is given by $\mathcal{R}$ and a typical action is denoted by $r$.
  \item[(iii)] The {\em substochastic} transition law is for all $x,y \in \R^d, t \geq 0, n \in \N_0$ and $C, D\in \mathcal{B}_d$ given by
\begin{eqnarray}
\nonumber && \tilde{Q}\big([0,t]\times C\times D|s,y,x,r\big)= \tilde{Q}\big([0,t]\times C\times D|y,r\big) \\
&=& \int_0^t \exp\big( -\Gamma^r(y,u)\big)  \int_A\lambda^A\big(\Phi^r(y,u),a\big)
\int_C Q(D-y') Q^A\big(dy' | \Phi^r(y,u),a\big) r_u(da)du
\end{eqnarray}
where $\Gamma^r(y,t) := \beta t+\int_0^t\int_A \lambda^A(\Phi^r(y,u),a) r_u(da) du$. Note that in case $\lambda^A \equiv \lambda$ we have $\tilde{Q}\big([0,\infty)\times \R^{2d}|y,r\big)=\frac{\lambda}{\beta+\lambda} < 1$.
 \item[(iv)] The  one-stage cost depends only on $y\in \R^d, r\in \mathcal{R}$ and is given by
 \begin{eqnarray}
\nonumber g(y, r) &:=& \Eop_{y}^\pi \left[ \int_0^{T_1} e^{-\beta t} \int_A c( \Phi^r(y, t), a) ~r_t(da) ~dt \right]\\
&=& \int_0^\infty \exp\big(-\Gamma^r(y,t)\big) \int_A c( \Phi^r(y, t), a) ~r_t(da) ~dt.
\end{eqnarray}
\end{enumerate}
\end{definition}

The last equation follows from the fact that the density of $T_1$ under $\Prob_{x,y}^\pi$ is given by
$$ f_{T_1}(y,t) = e^{-\Lambda^r(y,t)} \int_A \lambda^A\big( \Phi^r(y,t),a\big) r_t(da) $$
and with the help of Fubini's Theorem. In order to ease notation we still denote the corresponding POMDP by  $(S_n,\hat{Y}_n,X_n)$.
According to the Theorem of Ionescu Tulcea $Q_0(\cdot|x)$ together with the transition kernel $\tilde{Q}$ defines a probability measure $\tilde{\Prob}_x$. The difference between $\Prob_{x,y}$ and $\tilde{\Prob}_x$ is that $\tilde{\Prob}_x$ keeps track of discounting and is thus in general substochastic.
For the POMDP, policies are defined as history dependent relaxed control policies $\pi^D := ( \pi_0^D, \pi_1^D, \dots)$ with $\pi_n^D: \Hist_n \to \mathcal{R}$ measurable.

For a policy $\pi \in \Pi^D$ (we write $\pi$ instead of $\pi^D$ to ease notation) and an initial observation $x \in \R^d$  we define the cost of policy $\pi$ as
\begin{equation}
\tilde{J}(x, \pi) := \widetilde{\Eop}_x^{\pi} \left[ \sum_{k=0}^\infty  g(\hat{Y}_k, \pi_k(H_k)) \right]
\end{equation}
where $\widetilde{{\Eop}}_x$ is the expectation with respect to the probability measure $\tilde{\Prob}_x$.

The \emph{value function} of the discrete time control model gives the minimal cost under an initial observation $x \in \R^d$ and is defined as
\begin{equation}\label{eq:ValueFunctionControlModelTimeDiscrete}
\tilde{J}(x) := \inf_{\pi \in \Pi^D} \tilde{J}(x, \pi) \hspace{0.5cm} \forall ~x \in \R^d.
\end{equation}
The \emph{discrete time optimization problem} is then to find, for $x \in \R^d$, a policy $\pi^{\star } \in \Pi^D$ such that  we get
$\tilde{J}(x) = \tilde{J}(x,\pi^{\star }).$ The next lemma shows that this problem is equivalent to controlling the POPDMP in \eqref{eq:CostOfStrategy}.

\begin{lemma} \label{prop:EquivalenceTimeContToTimeDiscIOptProb} Let $x \in \R^d$ be an initial observation, $\pi^P \in \Pi^P$ a history dependent relaxed piecewise open loop control policy for the POPDMP and $\pi^D \in \Pi^D$ its corresponding discrete-time policy according to the Correspondence Lemma. Then, it holds
$$J(x, \pi^P) = \tilde{J}(x, \pi^D).$$
\end{lemma}

\begin{proof}
We obtain with the Correspondence Lemma:
\begin{eqnarray*}
&&  J(x,\pi^P) =  \\
&=&\int \Eop_{x,y}^{\pi^P} \left[ \int_0^{\infty} e^{-\beta t} \int_A c(Y_t, a) ~\pi_t(da) ~dt \right] Q_0(dy|x), \\
   &=&  \int \Eop_{x,y}^{\pi^P} \left[ \sum_{k=0}^\infty \int_{T_k}^{T_{k+1}} e^{-\beta t} \int_A c(Y_t, a) ~\pi_k^P(H_k,t-T_k)(da) ~dt \right] Q_0(dy|x) \\
   &=&  \int {\Eop}_{x,y}^{\pi^D} \left[ \sum_{k=0}^\infty \int_{T_k}^{T_{k+1}} e^{-\beta t} \int_A c(Y_t, a) ~\pi_k^D(H_k)(t-T_k)(da) ~dt \right] Q_0(dy|x) \\
   &=&  \int {\Eop}_{x,y}^{\pi^D} \left[ \sum_{k=0}^\infty e^{-\beta T_k}\Eop_{\hat{Y}_k}^{\pi^D}\left[ \int\limits_{T_k}^{T_{k+1}} \!\!\!\!e^{-\beta (t-T_k)} \int_A c(Y_t, a) ~\pi_k^D(H_k)(t-T_k)(da) ~dt \Big| H_k, \hat{Y}_k, T_k\right] \right] Q_0(dy|x)\\
    &=&  \int {\Eop}_{x,y}^{\pi^D} \left[ \sum_{k=0}^\infty e^{-\beta T_k} g(\hat{Y}_k, \pi_k^D(H_k))\right]  Q_0(dy|x)\\
   &=&   \widetilde{\Eop}_{x}^{\pi^D} \left[ \sum_{k=0}^\infty g(\hat{Y}_k, \pi_k^D(H_k)) \right]
\end{eqnarray*}
which is exactly the right hand side. Note that in the last sum there is no additional discount factor. The term $e^{-\beta T_k}$ which appears in the last but one equation is now part of the probability measure $\tilde{\Prob}_x^{\pi^D}$ (see Definition \ref{def:translawMDP}) which is substochastic.
\end{proof}

In the remaining section we explain how this POMDP can be transformed into a completely observable MDP which will then be solved in the next section.
We make some further simplifying assumptions. The first one implies that we later  get a finite dimensional filter for our problem.\\

{\em Assumption:}
\begin{description}
  \item[(B1)] There exists a finite subset $E^0\subset \R^d$ with $E^0=\{y^1,\ldots,y^d\}$ such that for all $y\in \R^d$ and $a\in A$: $Q^A(E^0 | y,a)=1$ and $Q_0$ is also concentrated on $E^0$.
  \item[(B2)] $Q_\epsilon$ has a bounded density $f_\epsilon$ with respect to some $\sigma$-finite measure $\nu$.
\end{description}

\vspace*{0.4cm}

Under Assumption (B1)-(B2) our substochastic transition law in Definition \ref{def:translawMDP} has a density with respect to the product of Lebesgue measure, $\nu$ and the counting measure given by
\begin{eqnarray*}
&&\tilde{q}(s,y',x|y,r) \\
&=& \exp\big( -\Gamma^{r}(y,s)\big)  f(x-y') \int_A\lambda^A\big(\Phi^r(y,s),a\big)  Q^A\big(y' | \Phi^r(y,s),a\big) r_s(da).
\end{eqnarray*}
In order to reduce problem \eqref{eq:ValueFunctionControlModelTimeDiscrete} to an MDP with complete observation we have to replace the unobservable state by its conditional distribution given the history so far.  The computation of this conditional distribution can be done recursively. This is a Bayesian updating procedure. The conditional distribution is also called \emph{filter}. In what follows we will introduce the updating-operator $\Psi: \mathcal{P}(E^0)\times \mathcal{R}\times \R_+\times \R^d\to \mathcal{P}(E^0) $ which maps the conditional distribution $\rho$ of the previous step, the relaxed control $r$ which is chosen  and the received new information (this is the time point of the jump $s$ and the observation $x$) onto the new conditional distribution. The updating operator essentially follows from  Bayes' formula. We will later show in Lemma \ref{lem:bayes} that the recursive computation which is done here really yields the conditional distribution of the unobservable state. The updating-operator is defined as
\begin{equation}
    \Psi(\rho,r,s,x)(y') := \frac{ \sum_{y\in E^0} \tilde{q}(s,y',x|y,r) \rho(y)}{\sum_{\hat{y}\in E^0} \sum_{y\in E^0} \tilde{q}(s,\hat{y},x|y,r) \rho(y) }.
\end{equation}
When we denote for the history $h_n=(x_0,r_0,s_1,x_1,\ldots ,r_{n-1},s_n,x_n)$ up to time $T_n$ the following distributions
\begin{eqnarray*}
  \mu_0(x_0) &:=& Q_0(\cdot|x_0), \\
  \mu_n(\cdot|h_n)=\mu_n(\cdot | h_{n-1},r_{n-1},s_n,x_n) &:=& \Psi\big(\mu_{n-1}(\cdot|h_{n-1}),r_{n-1},s_{n},x_{n}\big),
\end{eqnarray*}
then we obtain the necessary quantity to reduce the problem to an MDP with complete observation. The previous equation is also called {\em filter equation.}

\begin{definition}\label{def:filteredMDP}
Consider the following discrete-time filtered Markov Decision Model with complete observation:
\begin{enumerate}
  \item[(i)] The state space of this process is given by $\mathcal{P}(E^0)$. A typical state is denoted by $\rho$. The interpretation of $\rho$ is that it is the current  conditional probability of the unobservable state.
  \item[(ii)] The action space is given by $\mathcal{R}$. A typical action is denoted by $r$.
  \item[(iii)] The transition kernel $\hat{Q}$ from $ \mathcal{P}(E^0)\times\mathcal{R}$ to $\mathcal{P}(E^0)$ is for all $r\in\mathcal{R}$ and $ \rho\in \mathcal{P}(E^0)$ given by
\begin{eqnarray}
\hat{Q}(B | \rho, r) &=& \int_{\R_+} \int_{\R^d} \sum_{y\in E^0} 1_B\big(\Psi(\rho,r,s,x)\big) \tilde{q}^{SX}(s, x|y,r)\nu(dx)ds\rho(y)
\end{eqnarray}
where $\tilde{q}^{SX}(s, x|y,r) :=  \sum_{y'\in E^0} \tilde{q}(s,y',x|y,r)$.
 \item[(iv)] The  one-stage cost  is given by
 \begin{equation}\label{eq:g}
\hat{g}(\rho, r) := \sum_{y\in E^0} g(y,r) \rho(y).
\end{equation}
\end{enumerate}
\end{definition}
The corresponding filtered MDP is denoted by $(\mu_n)$. Policies $\pi=(f_0,f_1,\ldots)$ are here defined as Markovian decision rules $f:\mathcal{P}(E^0)\to \mathcal{R}$. We denote by $\Pi$ the set of all decision rules. Every $\pi =(f_0,f_1,\ldots)\in \Pi^\infty$ can be seen as a special policy $\pi^D\in\Pi^D$ by setting
\begin{equation}\label{eq:policy} \pi_n^D(h_n):= f_n(\mu_n(\cdot|h_n)).\end{equation}
Note that in MDP theory it is well-known that we can restrict the optimization to Markovian policies (see \cite{hin}, Theorem 18.4).

An initial distribution $\rho$ on $\mathcal{P}(E^0)$ together with the transition kernels $\hat{Q}$ define a probability measure $\hat{\Prob}_\rho$.
We denote the \emph{cost of policy} $\pi\in\Pi^\infty$ under an initial distribution $\rho \in \mathcal{P}(E^0)$ by
\begin{equation}\label{eq:CostOfStrategy2}
V(\rho,  \pi) := {\hat{\Eop}}_{\rho}^\pi \left[ \sum_{n=0}^{\infty} \hat{g}\big(\mu_n,f_n(\mu_n)\big) \right].
\end{equation}
The \emph{value function} of the control model gives the minimal cost under an initial distribution $\rho \in \mathcal{P}(E^0)$ and is defined as
\begin{equation}\label{eq:ValueFunctionfilter}
V(\rho) := \inf_{\pi \in \Pi^\infty} V(\rho, \pi) \hspace{0.5cm} \mbox{ for all } ~\rho \in \mathcal{P}(E^0).
\end{equation}
The \emph{optimization problem} is then to find, for $\rho \in \mathcal{P}(E^0)$, a policy $\pi^\star \in \Pi^\infty$ such that  we get
\begin{equation}
V(\rho) = V(\rho,\pi^\star).\end{equation}

\begin{lemma}\label{lem:bayes}  Let $x \in \R^d$ be an initial observation, $\pi \in \Pi^\infty$ and $\pi^D$ given by \eqref{eq:policy}. Then, it holds
$$V(Q_0(\cdot|x), \pi) = J(x, \pi^D).$$
\end{lemma}

\begin{proof}
Similar proofs can be found  in \cite{br11},  Theorem 5.3.2.\ or \cite{br17} Theorem 3.2. We first show that for any measurable $v : \mathcal{H}_n\times \R^d \to\R$ (provided the expectations exist)
\begin{eqnarray}\nonumber&&\tilde{\Eop}_x^{\pi^D}\Big[ v(X_0,R_0,S_1,X_1,\ldots ,R_{n-1},S_n,X_n,\hat{Y}_n)\Big] \\
\label{eq:bayes1}&=& \hat{\Eop}_{Q_0}^{\pi}\Big[ v'(X_0,R_0,S_1,X_1,\ldots ,R_{n-1},S_n,X_n,\mu_n)\Big]
\end{eqnarray}
where  $R_n := f_n(\mu_n(\cdot|H_n))$ and $v'(h_n,\rho) := \sum_{y\in E^0} v(h_n,y) \rho(y)$. This can be shown by induction on $n$. For $n=0$ we have
\begin{eqnarray*}
  \tilde{\Eop}_x^{\pi^D}\Big[ v(X_0,\hat{Y}_0)\Big]  &=& \sum_{y\in E^0} v(x,y)  Q_0(y|x) \\
  \hat{\Eop}_{Q_0(\cdot|x)}^{\pi}\Big[ v'(X_0,\mu_0)\Big] &=& \sum_{y\in E^0} v(x,y)  Q_0(y|x)
\end{eqnarray*}
so obviously both sides are equal. Now suppose the statement is true for $n-1$ and fix $H_{n-1}=h_{n-1}$. The left-hand side of \eqref{eq:bayes1} can be written as
\begin{eqnarray*}
   &&\tilde{\Eop}_x^{\pi^D}\Big[ v(h_{n-1},R_{n-1},S_n,X_n,\hat{Y}_n)\Big]  \\
   &=& \sum_{y_{n-1}} \mu_{n-1}\big(y_{n-1} | h_{n-1}\big) \int_{\R^d}\int_{\R_+}\sum_{y_n} \tilde{q}\big(s_n,y_n,x_n|y_{n-1},\pi_{n-1}^D(h_{n-1})\big) \\
   && \hspace*{1cm} v\big(h_{n-1}, \pi_{n-1}^D(h_{n-1}),s_n,x_n,y_n)  d s_n \nu(dx_n).
\end{eqnarray*}
The right-hand side can be written as (where we use $\mu_n = \Psi(\mu_{n-1},f_{n-1}(\mu_{n-1}),s_n,x_n)$ in the second equation)
\begin{eqnarray*}
   &&\hat{\Eop}_{Q_0}^{\pi}\Big[ v'(h_{n-1},R_{n-1},S_n,X_n,\mu_n)\Big]  \\
   &=& \sum_{y_{n-1}} \mu_{n-1}\big(y_{n-1} | h_{n-1}\big) \int_{\R_+}\int_{\R^d} \tilde{q}^{SX}\big(s_n,x_n|y_{n-1},f_{n-1}(\mu_{n-1})\big) \\
   && \hspace*{1cm} v'\big(h_{n-1}, f_{n-1}(\mu_{n-1}),s_n,x_n,\mu_n(\cdot|h_{n-1},f_{n-1},s_n,x_n)\big)   \nu(dx_n) ds_n\\
   &=& \sum_{y_{n-1}} \mu_{n-1}\big(y_{n-1} | h_{n-1}\big) \int_{\R_+} \int_{\R^d}\tilde{q}^{SX}\big(s_n,x_n|y_{n-1},f_{n-1}(\mu_{n-1})\big) \\
    && \hspace*{1cm} \sum_{y_n} v\big(h_{n-1}, f_{n-1}(\mu_{n-1}),s_n,x_n,y_n) \\
    && \hspace*{1cm} \frac{ \sum_{y} \tilde{q}(s_n,y_n,x_n|y,f_{n-1}(\mu_{n-1})) \mu_{n-1}(y|h_{n-1}) }{\sum_{y'}\sum_{y} \tilde{q}(s_n,y',x_n|y,f_{n-1}(\mu_{n-1})) \mu_{n-1}(y|h_{n-1})}   \nu(dx_n) ds_n.
\end{eqnarray*}
Note that we have
\begin{eqnarray*}
   && \sum_{y'}\sum_{y} \tilde{q}(s_n,y',x_n|y,f_{n-1}(\mu_{n-1})) \mu_{n-1}(y|h_{n-1})\\
   &=& \sum_{y} \tilde{q}^{SX}(s_n,x_n|y,f_{n-1}(\mu_{n-1})) \mu_{n-1}(y|h_{n-1}).
\end{eqnarray*}
Applying Fubini's Theorem to interchange the integrals we see that
\begin{eqnarray*}
   && \hat{\Eop}_{Q_0}^{\pi}\Big[ v'(h_{n-1},R_{n-1},S_n,X_n,\mu_n)\Big]  \\
   &=&   \int_{\R^d}\int_{\R_+} \sum_{y_n} v\big(h_{n-1}, f_{n-1}(\mu_{n-1}),s_n,x_n,y_n)\\
    && \sum_{y} \tilde{q}(s_n,y_n,x_n|y,f_{n-1}(\mu_{n-1})) \mu_{n-1}(y|h_{n-1})  ds_n \nu(dx_n) .
\end{eqnarray*}
and thus both sides are equal. When we choose
$$ v(H_n,\hat{Y}_n) = g\big(\hat{Y}_n,\pi_n^D(H_n)\big)$$ we obtain that
$$ \tilde{\Eop}_x^{\pi^D}\Big[ g\big(\hat{Y}_n,\pi_n^D(H_n)\big)\Big] = \hat{\Eop}_{Q_0}^{\pi}\Big[ \hat{g}\big(\mu_n,f_n(\mu_n(\cdot|H_n))\big)\Big]$$
where we use definition \eqref{eq:g} of $\hat{g}$ on the right-hand side which implies the statement.
\end{proof}

\begin{remark}
When we choose $v=1_{B\times C}$ in the previous proof, we obtain
\begin{eqnarray*}
   && \tilde{\Prob}_x^{\pi^D} \Big( (X_0,R_0,S_1,X_1,\ldots ,R_{n-1},S_n,X_n)\in B,  \hat{Y}_n\in C\Big) \\
   &=& \hat{\Eop}_{Q_0}^{\pi} \Big[ 1_B(X_0,R_0,S_1,X_1,\ldots ,R_{n-1},S_n,X_n)\cdot \mu_n(C| X_0,R_0,S_1,X_1,\ldots ,R_{n-1},S_n,X_n)\Big]
\end{eqnarray*}
which implies that $\mu_n$ is a conditional $\tilde{\Prob}_x^{\pi^D}$-distribution of $\hat{Y}_n$ given the previous history  \linebreak $(X_0,R_0,S_1,X_1,\ldots ,R_{n-1},S_n,X_n)$.
\end{remark}

\begin{remark}
If $\lambda^A$ and $Q^A$ are not controlled i.e. do not depend on $A$ we obtain the following special substochastic transition kernel:
\begin{eqnarray*}
&&\tilde{q}(s,y',x|y,r) \\
&=& \exp\Big( -\beta t - \int_0^t \lambda(\Phi^r(y,s))ds\Big) f(x-y') \lambda\big(\Phi^r(y,t)\big)  Q\big(y' | \Phi^r(y,t)\big)
\end{eqnarray*}
i.e. the updating-operator $\Psi$ depends on $r$ only through $\Phi^r(y,\cdot)$. This observation will be crucial later on (see Remark \ref{rem:Quncontrolled}).
\end{remark}

\section{Optimality Equation and Existence of Optimal Policies}\label{sec:oe}
In this section we will formulate our main theorem which states existence of an optimal policy for the original problem \eqref{eq:ValueFunctionControlModelTimeContinuous} and provides an optimality equation for the value function. The critical point here is to find the right continuity and compactness conditions in order to show the existence of optimal policies. In particular we have to replace the filter by a regularized version in the general case. Thus, we first  make the following assumption.\\

{\em Assumption}:
\begin{description}
  \item[(C4)]  The mapping $r \mapsto \Phi^r(y,t)$ is continuous for all $y \in E^0$ and $t \geq 0$.
  \item[(C5)] The cost function $c:\R^d \times A \to \R_+$ is lower semi-continuous with respect to the product topology.
\end{description}

The proof of the following five lemmas can be found in the appendix.

\begin{lemma}\label{le:LambdaRIsContinuous}
Under Assumptions (C1),(C2),(C4), the mapping $r \mapsto \Gamma^r(y,t)$ is continuous for all $y \in E^0$ and $t \geq 0$.
\end{lemma}

\begin{lemma}\label{theo:GPrimeLowerSemiContinuous} Under Assumptions (C1),(C2),(C4),(C5) the one-step cost function $(\rho,r) \mapsto \hat{g}(\rho,r)$ of the derived filtered model is lower semi-continuous.
\end{lemma}

\begin{lemma}\label{lem:Continuous2} Under Assumptions (C1)-(C4),(B1),(B2) we have that $r\mapsto \int_{\R_+} \tilde{q}(s,y',x|y,r)ds$ is continuous for all $y,y' \in E^0$ and $x\in \R^d$.
\end{lemma}

In order to prove the continuity of the transition kernel of the filtered MDP we use the following regularization of the filter. Let $h_\sigma :\R\to \R, \sigma>0$ be a regularization kernel , i.e.
\begin{itemize}
\item[(i)] $h_\sigma(t)\ge 0$ for all $t\in\R$,
\item[(ii)] $\int_\R h_\sigma(t)dt =1$,
\item[(iii)] $\lim_{\sigma\downarrow 0} \int_{-a}^a h_\sigma(t) dt =1$ for all $a>0$.
\end{itemize}
The function $h_\sigma$ approximates the Dirac measure in point zero. For the general existence result we use a regularized filter of the form $\hat{\Psi}:\mathcal{P}(E^0)\times \mathcal{R}\times \R_+\times \R^d\to \mathcal{P}(E^0) $
\begin{equation}
    \hat{\Psi}(\rho,r,s,x)(y') := \frac{\int_\R \sum_{y\in E^0} \tilde{q}(u,y',x|y,r) \rho(y)h_\sigma(s-u)du}{\sum_{\hat{y}\in E^0}\int_\R \sum_{y\in E^0} \tilde{q}(u,\hat{y},x|y,r) \rho(y)h_\sigma(s-u)du }.
\end{equation}
Note that we have $\lim_{\sigma\downarrow 0} \hat{\Psi}=\Psi$ (see e.g. \cite{bre14}, Theorem 1.1.7).

\begin{lemma}\label{lem:Continuous} Under Assumptions (C1)-(C4),(B1),(B2) we have that  $(\rho,r)\mapsto \hat{\Psi}(\rho,r,x,u)$ is continuous for all $ x\in \R^d$.
\end{lemma}

Finally we obtain:


\begin{lemma}\label{lem:qprimeweaklycont}
Under all Assumptions  (C1)-(C4), (B1),(B2) the stochastic transition kernel $\hat{Q}$ in Definition \ref{def:filteredMDP}   where we replace $\Psi$ by the regularized filter $\hat{\Psi}$ is weakly continuous.
\end{lemma}

In what follows we will always assume that the regularized filter version is used in the definition of $\hat{Q}$. The next step is to define the following function space
$$ \mathcal{C}_{lsc}^+ := \{ v: \mathcal{P}(E^0) \to [0,\infty] : v \mbox{  is lower semi-continuous }\}$$
and the following operators for $v\in \mathcal{C}_{lsc}^+, \rho\in  \mathcal{P}(E^0), r\in \mathcal{R}  $ and $f\in \Pi$:
\begin{eqnarray*}
  (Lv)(\rho,r) &:=& \hat{g}(\rho,r) + \int_{\mathcal{P}(E^0)} v(\rho') \hat{Q}(d\rho'|\rho,r), \\
   (T_fv)(\rho) &:=& \hat{g}(\rho,f(\rho)) + \int_{\mathcal{P}(E^0)} v(\rho') \hat{Q}(d\rho'|\rho,f(\rho)), \\
   (Tv)(\rho)&:=& \inf_{r\in\mathcal{R}} (Lv)(\rho,r).
\end{eqnarray*}

Our previous results lead now to the following observation:

\begin{theorem}\label{theo:lsc}
Under all Assumptions  (C1)-(C5), (B1),(B2) we have that
\begin{itemize}
  \item[a)] $T : \mathcal{C}_{lsc}^+ \to \mathcal{C}_{lsc}^+$.
  \item[b)] For all $v\in \mathcal{C}_{lsc}^+$ there exists an $f^*\in\Pi$ such that
  $$ (Tv)(\rho)= \inf_{r\in\mathcal{R}} (Lv)(\rho,r)= (Lv)(\rho,f^*(\rho)).$$
  \item[c)] For all $v,w\in \mathcal{C}_{lsc}^+$ with $v\le w$ we obtain $Tv\le Tw$.
\end{itemize}
\end{theorem}

\begin{proof}
The proof of this theorem is rather standard. If we first choose $v$ to be continuous and bounded we obtain by Lemma \ref{lem:qprimeweaklycont} that
$$ (\rho,r)\mapsto \int_{\mathcal{P}(E^0)} v(\rho') \hat{Q}(d\rho'|\rho,r)$$
is continuous and bounded. Thus using the same line of arguments as in the proof of Lemma \ref{theo:GPrimeLowerSemiContinuous} we obtain that the same mapping is lower semi-continuous when we plug in a lower semi-continuous function $v$. Since the sum of lower semi-continuous functions is again lower semi-continuous we get with Lemma \ref{theo:GPrimeLowerSemiContinuous} that
$$ (\rho,r) \mapsto \hat{g}(\rho,r)+\int_{\mathcal{P}(E^0)} v(\rho') \hat{Q}(d\rho'|\rho,r)= (Lv)(\rho,r)$$
is lower semi-continuous. We can now use a classical measurable selection theorem for part b) (see e.g. Proposition 7.33 in \cite{BertShrev}) and part a) follows as in Proposition 2.4.3 in \cite{br11}, see also section 3.3 of \cite{hll} or Propositions 7.31 and 7.33 in \cite{BertShrev}. Part c) is obvious.
\end{proof}

\begin{remark}\label{rem:Quncontrolled}
Note  that the existence of a minimizer $f^*\in \Pi$ in Theorem \ref{theo:lsc} b) cannot be shown in general if we take the original filter $\Psi$. In this case examples can be constructed where the filter is not continuous  (for a discussion and the example see \cite{lange17} Section 3.2.2). The crucial point here is that the single action which is applied at the jump time point enters the filter. This effect is incompatible with the Young toplogy. It does not occur when $\lambda^A$ and $Q^A$ are uncontrolled.  Indeed
Lemma \ref{lem:qprimeweaklycont}  holds true for the original filter  $\Psi$  in case $\lambda^A$, $Q^A$ are not controlled. In this case we have
\begin{equation*}
    \Psi(\rho,r,s,x)(y') := \frac{\sum_{y\in E^0} \tilde{p}(s,y',x|y,r) \rho(y)}{\sum_{\hat{y}\in E^0} \sum_{y\in E^0} \tilde{p}(s,\hat{y},x|y,r) \rho(y) }.
\end{equation*}
with
$$\tilde{p}(t,y',x|y,r)=\exp\Big( -\int_0^t \lambda(\Phi^r(y,s))ds\Big) f(x-y') \lambda(\Phi^r(y,t)) Q(y'|\Phi^r(y,t)). $$
So $\Psi$ depends on $r$ only through $\Phi^r(y,\cdot)$ which is continuous by Assumption (C4).
\end{remark}

In order to derive the optimality equation we need to consider the $n$-stage version of the optimization problems. Thus we define for a policy $\pi\in\Pi^\infty$ and the function $ \underline{0}\in   \mathcal{C}_{lsc}^+$ which is identical to zero, the following value functions:
\begin{eqnarray*}
  V_n(\rho,\pi) &:=& T_{f_0}\ldots T_{f_{n-1}} \underline{0} \\
  V_n(\rho) &:=& \inf_{\pi\in\Pi^\infty} V_n(\rho,\pi) = T^n  \underline{0}.
\end{eqnarray*}
Note that $V_n(\rho,\pi)$ is exactly the expected cost of policy $\pi$ until jump time $T_n$. By general MDP techniques (see e.g. \cite{br11}, chap.\ 2) we obtain the last equation $V_n(\rho) = T^n  \underline{0}$ which also implies that $V_n=TV_{n-1}$. Since the cost function is non-negative we obtain by monotone convergence that the following limits exist:
\begin{eqnarray*}
  V(\rho,\pi) &:=& \lim_{n\to\infty} V_n(\rho,\pi) \\
  V_\infty(\rho)  &:=& \lim_{n\to\infty} V_n(\rho).
\end{eqnarray*}
By definition we get that $V(\rho) = \inf_{\pi\in\Pi^\infty} V(\rho,\pi)$. From Theorem \ref{theo:lsc} it follows that $V_n\in \mathcal{C}_{lsc}^+$ because $\underline{0}\in\mathcal{C}_{lsc}^+$ and hence also $V_\infty \in \mathcal{C}_{lsc}^+$. Moreover we immediately obtain by monotonicity that $V(\rho,\pi) \ge V_n(\rho,\pi)$ for all $\pi\in\Pi^\infty$ which then implies $V(\rho) \ge V_n(\rho)$ and with $n\to\infty$ that $V(\rho) \ge V_\infty(\rho)$. The main result of this section is the following:

\begin{theorem}\label{theo:VTV}
Under all Assumptions  (C1)-(C4), (B1),(B2) we have that
\begin{itemize}
  \item[a)] $TV_\infty =V_\infty$.
  \item[b)] $V_\infty = V.$
  \item[c)] There exists an $f^*\in\Pi$ with $TV=T_{f^*}V$ and the stationary policy $(f^*,f^*,\ldots)$ is optimal for problem \eqref{eq:ValueFunctionfilter}. The optimal policy for the original problem \eqref{eq:ValueFunctionControlModelTimeContinuous} is thus $(\pi_0^P,\pi_1^P,\ldots)$ with
      \begin{eqnarray*}
        \pi_0^P(x,t) &=& f^\star\big(Q_0(\cdot|x)\big)(t),\quad x\in \R^d \\
        \pi^P_n(h_n,t) &=& f^\star\big(\mu_n(\cdot|h_n)\big)(t), \quad h_n\in H_n.
      \end{eqnarray*}
\end{itemize}
\end{theorem}

\begin{proof}
\begin{itemize}
  \item[a)] Due to monotonicity of $V_n$ and the $T$-operator we obtain $V_n\le TV_\infty$ for all $n\in\N$. For $n\to\infty$ we obtain $V_\infty \le TV_\infty$. It remains to prove $V_\infty \ge TV_\infty$. Since $T^n  \underline{0} \in \mathcal{C}_{lsc}^+$ we know by Theorem \ref{theo:lsc} that there exist decision rules $f_k^*$ such that $T^n \underline{0} = T_{f_0^*}\ldots T_{f_{n-1}^*} \underline{0} $. Now fix $\rho\in \mathcal{P}(E^0)$ and define $r^n := f_n^*(\rho)$. Then $(r^n) \subset \mathcal{R}$ and since $\mathcal{R}$ is compact there exists a converging subsequence $\lim_{k\to\infty}r^{n_k}= r\in\mathcal{R}.$ This implies by monotonicity for an arbitrary index  $n_k$ and all $n\le n_k$:
      $$ V_\infty(\rho) \ge (T^{n_k+1} \underline{0})(\rho) = (L T^{n_k} \underline{0})(\rho,r^{n_k}) \ge (LT^n\underline{0})(\rho,r^{n_k}).$$
      Since $T^n \underline{0} \in \mathcal{C}_{lsc}^+$ the mapping $r\mapsto (LT^n \underline{0})(\rho,r)$ is lower semi-continuous. Thus we obtain by definition of this property that
      $$   V_\infty(\rho) \ge  \lim_{k\to\infty} (L T^{n} \underline{0})(\rho,r^{n_k}) \ge (LT^n\underline{0})(\rho,r).$$
      And with $n\to\infty$ we obtain by monotone convergence
       $$   V_\infty(\rho) \ge  (L V_\infty)(\rho,r) \ge (T V_\infty)(\rho)$$
       which finally implies the statment.
  \item[b)] Since $V\ge V_\infty$ it is sufficient to prove $V\le V_\infty$. By part a) we know that $V_\infty \ge TV_\infty$. Since $ V_\infty\in \mathcal{C}_{lsc}^+$ we know by Theorem \ref{theo:lsc} that there exists a decision rule $f^*$ such that $TV_\infty = T_{f^*}V_\infty$. Iterating this equation yields:
      $$ V_\infty \ge   T_{f^*}V_\infty =  T_{f^*}^nV_\infty \ge  T_{f^*}^n\underline{0}.$$ with $n\to\infty$ we obtain
      $$ V_\infty(\rho) \ge V(\rho,f^{*\infty}) \ge V(\rho)$$
      which implies the statement.
  \item[c)] This follows from the proof of part b) and the use of the Correspondence Lemma \ref{theo:CorrespondenceTheorem}.
\end{itemize}
\end{proof}

Besides the existence of optimal policies Theorem \ref{theo:VTV} presents a numerical way of computing the value function $V$. Since $V=V_\infty$ according to part b) we can use \emph{value iteration} to approximate $V$, i.e. we can start with $V_0 = \underline{0}$ and compute $V_n =TV_{n-1}$ for large $n$.

Since we use the regularized filter with a regularization kernel $h_\sigma$, the optimal policy depends on $\sigma$. For nice problems we expect convergence of the optimal policies for $\sigma\downarrow 0$ to an optimal policy for the original problem. A general theorem which guarantees this is as follows. Fix state $\rho\in  \mathcal{P}(E^0)$ and suppose that $\sigma_n\downarrow 0$ for $n\to\infty$. We denote the value function which corresponds to $\sigma_n$ by $V^n$. Let
$$ \mathcal{R}_n := \{r\in \mathcal{R}\; |\; (LV^{n})(\rho,r) = (TV^n)(\rho) \}$$ for $n\in\N\cup \{\infty\}$ be the set of maximum points of the value function $V^n$ in state $\rho$. The value function $V^\infty$ corresponds to the problem with original filter $\Psi$. Further let
$$ Ls \mathcal{R}_n := \{ r \in \mathcal{R}\;  |\;  r \mbox{ is an accumulation point of } (r^n) \mbox{ with } r^n\in\mathcal{R}_n\}. $$
The next theorem follows from Theorem A.1.5 in \cite{br11}.

\begin{theorem}\label{theo:convsigma}
Suppose there exists a sequence $\delta_m\downarrow 0$ for $m\to\infty$ such that $(LV^{n})(\rho,r) \ge (LV^{m})(\rho,r)+\delta_m$ for all $n\ge m$, i.e. the sequence $(LV^n)(\rho,r)$ is weakly increasing. Then $\emptyset \neq Ls \mathcal{R}_n \subset \mathcal{R}_\infty$.
\end{theorem}

The interpretation of Theorem \ref{theo:convsigma} is as follows: Fix $\rho\in \mathcal{P}(E^0)$ and take $\sigma_n\downarrow 0$. Suppose the sequence $(LV^n)(\rho,r)$ is weakly increasing. The sequence of  optimal relaxed controls $(r^n)$ in state $\rho$ has at least one accumulation point and every accumulation point is an optimal relaxed control in the original model with non-regularized filter.

\section{Application}\label{sec:applic}
In this section we illustrate our approach by a simple example. The task is to steer a particle which moves on the real line into a target zone. At random time points the particle jumps into one of a finite number of states. However, the position of the particle cannot be observed. The only information we have is that after the random jump time points a noisy signal is received.

We consider this problem as a POPDMP where we specify the following data:
\begin{itemize}
\item[(i)] The state space is assume to be $\R$ and the action space is assumed to be $A:=[-1;1]$. Here $a\in A$ refers to the speed with which the particle is moved into one of the two directions. Obviously $A$ is compact, hence (C1) is satisfied.
\item[(ii)] The set of possible post jump states is assumed to be $E^0:=\{-2,0,2\}$ and we set $y^1:=-2, y^2:=0$ and $y^3:=2$. Thus (B1) is valid.
\item[(iii)] The controlled drift is given by \begin{equation}\label{eq:PDEDefinedDriftex}
\frac{d}{dt} \Phi^r(y,t) = \int_A a r_t(da), \hspace{0.5cm} \Phi(y,0)=y.
\end{equation} this implies (C4).
\item[(iv)] We set $\lambda^A \equiv 1$ and $\beta:=1$, i.e. the transition rate is uncontrolled and the discount rate is equal to one. Hence (C2) is satisfied.
\item[(v)] The jump transition kernel $Q^A$ is also uncontrolled and specified as follows (see also Figure \ref{fig:Anwendungsbeispiel}):
$$
Q^A(\cdot|y) := \left\{\begin{array}{lr}
\delta_{y^1}(\cdot), &y\leq-2\\[2ex]
(-3-2y) \cdot \delta_{y^1}(\cdot)  + (4+2y)\cdot\delta_{y^2}(\cdot) , &-2 <y<-\frac{3}{2}\\[2ex]
\delta_{y^2}(\cdot), &-\frac{3}{2} \leq y \leq \frac{3}{2}\\[2ex]
(4-2y) \cdot\delta_{y^2}(\cdot) + (2y-3)\cdot\delta_{y^3}(\cdot) , &\frac{3}{2}<y<2\\[2ex]
\delta_{y^3}(\cdot), &2\leq y,
\end{array}\right.
$$
where $\delta_x$ is the Dirac measure on point $x\in E^0.$ Note that $Q^A$ is weakly continuous, hence (C3) holds.
\item[(vi)] The cost function is independent of $a$ and given by (see also Figure \ref{fig:Anwendungsbeispiel}):
$$
c(y) := \left\{\begin{array}{lr}
10, &y\leq -2\\[2ex]
-30-20y, &-2<y-\frac{3}{2}\\[2ex]
0, &-\frac{3}{2}\leq y \leq \frac{3}{2}\\[2ex]
20y-30, &\frac{3}{2}<y<2\\[2ex]
10, &y\geq 2.
\end{array}
\right.
$$
Note that $c$ is continuous which implies (C5).
\item[(vii)] For the density of the signal we take the discrete density  $f_\epsilon(-1)= f_\epsilon(0)=f_\epsilon(1)=\frac13$. Hence (B2) is satisfied.
\end{itemize}

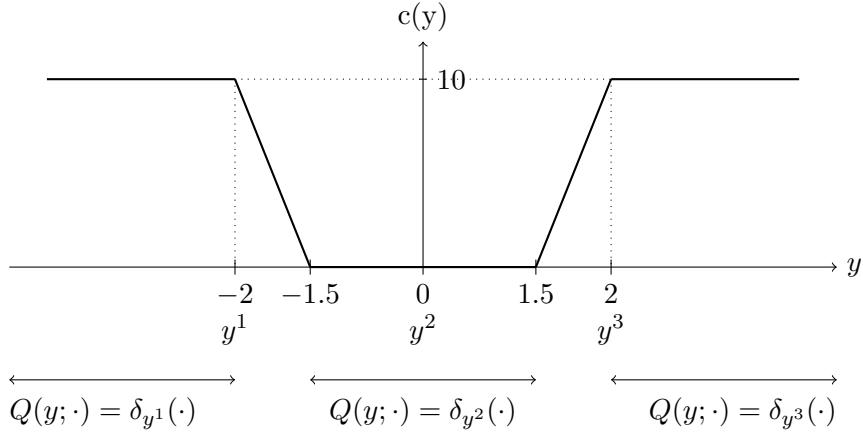
\begin{figure}
\begin{center}
\begin{tikzpicture}[scale=0.5]
    \draw[->] (-11,0) -- (11,0) node[anchor=west] {$y$};
    \draw[->] (0,-0.2)node[anchor=north] {$0$} -- (0,6) node[anchor=south] {c(y)};
		\draw(0,-1)node[anchor=north]{$y^2$};
		\draw[thick, color=black] (-5,5)--(-3,0);
		\draw[thick, color=black] (3,0)--(5,5);
		\draw[thick, color=black] (-3,0)--(3,0);
		\draw[thick, color=black] (-10,5)--(-5,5);
		\draw[thick, color=black] (5,5)--(10,5);
		\draw[dotted] (-5,5)--(5,5);
		\draw[dotted] (-5,5)--(-5,0);
		\draw(-5,0.2)--(-5,-0.2)node[anchor=north]{$-2$};
		\draw(-5,-1)node[anchor=north]{$y^1$};
		\draw(-3,0.2)--(-3,-0.2)node[anchor=north]{$-1.5$};
		\draw(3,0.2)--(3,-0.2)node[anchor=north]{$1.5$};
		\draw[dotted] (5,5)--(5,0);
		\draw(5,0.2)--(5,-0.2)node[anchor=north]{$2$};
		\draw(5,-1)node[anchor=north]{$y^3$};
		\draw(-0.1,5)--(0.1,5)node[anchor=west]{$10$};
		\draw[<->] (-3,-3) -- (3,-3);
		\draw[<->] (5,-3) -- (11,-3);
		\draw[<->] (-11,-3) -- (-5,-3);
		\draw(0,-3.2)node[anchor=north]{$Q(y;\cdot)=\delta_{y^2}(\cdot)$};
		\draw(8.5,-3.2)node[anchor=north]{$Q(y;\cdot)=\delta_{y^3}(\cdot)$};
		\draw(-8.5,-3.2)node[anchor=north]{$Q(y;\cdot)=\delta_{y^1}(\cdot)$};
    \end{tikzpicture}
		\end{center}
\caption{Cost function and transition kernel in concrete application example.}\label{fig:Anwendungsbeispiel}
\end{figure}

First note that since only
$  \int_A a r_t(da)$
enters the equations we can restrict to deterministic controls. We still denote them by $r$ and consider $r_t\in A$ instead of $r_t \in \mathcal{P}(A)$ which is a slight abuse of notation. The updating operator $\Psi$ in this case reads
$$ \Psi(\rho,r,s,x)(y^j) = \frac{f(x-y^j)\sum_y Q(y^j | y+\int_0^s r_udu) \rho(y)}{\sum_{\hat{y}}f(x-\hat{y})\sum_y Q(\hat{y} | y+\int_0^s r_udu) \rho(y)}.$$
Since $\lambda^A=\lambda$ and $Q^A=Q$ are uncontrolled we do not have to consider the regularized filter.
The one-stage reward is given by
$$\hat{g}(\rho,r) = \sum_y \rho(y) \int_0^\infty e^{-2t} c\big(y+ \int_0^t r_sds\big)dt$$
and finally the transition kernel is for a measurable function $v: \mathcal{P}(E^0) \to\R$ given by
$$ \int v(\rho') \hat{Q}(d\rho'|\rho,r) =  \frac{1}{3}\int_0^\infty e^{-2t} \sum_{d=-1}^1 \sum_{y'}v\big(\Psi(\rho,r,t,y'+d)\big)\sum_y Q(y' | y+\int_0^t r_sds)\rho(y) dt.$$
The optimization problem is
\begin{equation}\label{eq:optex} J(x) := \inf_{\pi\in\Pi^P} \int \Eop_{x,y}^\pi \Big[\int_0^\infty e^{-t} c(Y_t) dt \Big] Q_0(dy|x)\end{equation}
and the corresponding filtered MDP is defined by the $T$-operator which in this example reads
\begin{eqnarray}
\nonumber  (Tv)(\rho) &=& \inf_{r\in\mathcal{R}} \Bigg\{  \int_0^\infty e^{-2t} \Big[ \sum_y  \rho(y)  c\Big(y+ \int_0^t r_s ds\Big) \\
&& \hspace*{1cm} +    \frac{1}{3} \sum_{d=-1}^1 \sum_{y'}v\big(\Psi(\rho,r,t,y'+d)\big)\sum_y Q(y' | y+\int_0^t r_s ds)\rho(y)\Big] dt\Bigg\}.
\end{eqnarray}

Since all assumptions of Theorem \ref{theo:VTV} are satisfied we obtain in this example.

\begin{lemma}\label{theo:exlsc}
 In this POPDMP there exists an $f^*\in\Pi$ with $TV=T_{f^*}V$ and the stationary policy $(f^*,f^*,\ldots)$ is optimal for the filtered MDP. The optimal policy for the original problem \eqref{eq:optex} is thus $(\pi_0^P,\pi_1^P,\ldots)$ with
      \begin{eqnarray*}
        \pi_0^P(x,t) &=& f^\star\big(Q_0(\cdot|x)\big)(t),\quad x\in \R^d \\
        \pi^P_n(h_n,t) &=& f^\star\big(\mu_n(\cdot|h_n)\big)(t), \quad h_n\in H_n.
      \end{eqnarray*}

\end{lemma}

In this example we have also computed the value function and the optimal policy numerically by value iteration. The value function $V$ as a function of $\rho_1\in(0,1)$ and $\rho_3\in(0,1-\rho_1)$ can be seen in Figure \ref{fig:VF01}. The optimal policy turned out to always use one of the values $\{-1,0,1\}$. More precisely we obtain
\begin{equation}\label{eq:optPolicyPartObsAppl2}
\pi_n(h_n,t):= \left\{\begin{array}{cc}
 1_{\{t \leq \frac{1}{2}\}} & \mbox{ if } \mu_n^1\ge  \mu_n^3,\\
 - 1_{\{t \leq \frac{1}{2}\}}  & \mbox{ if } \mu_n^1\le \mu_n^3.
 \end{array}\right.
\end{equation}
Recall that $\mu_n(h_n)$ is the recursively calculated conditional distribution on $\mathcal{P}(E^0)$.

\begin{figure}
    \centering
     \includegraphics[height=8cm]{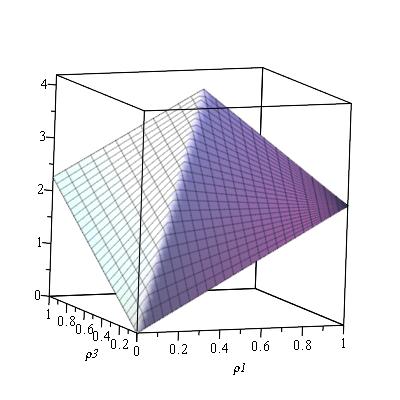} 
   \caption{Value function $V(\rho_1,\cdot,\rho_3)$ }
    \label{fig:VF01}
\end{figure}

\section{Appendix}\label{sec:app}
\subsection{Young Topology}
The Young topology is metrizable and convergence can be characterized as follows (for a proof see e.g. \cite{lange17} Lemma A.21):
\begin{definition}\label{YOUNG-4} Let $(r^n)_{n \in \N}$ be a sequence in $\mathcal{R}$ and $r \in \mathcal{R}$. Then
\begin{equation} \lim_{n\to\infty} r^n = r \hspace{0.7cm}
\Longleftrightarrow \hspace{0.7cm} \lim_{n\to\infty} \int_0^\infty \int_A \psi(t,a) ~r_t^n(da) ~dt = \int_0^\infty \int_A \psi(t,a) ~r_t(da) ~dt \hspace{0.5cm}
\end{equation}
for all $\psi: \R_+\times A\to\R $ which are measurable in the first component, continuous and bounded in the second component and satisfy
$$ \int_0^\infty \sup_{a\in A} |\psi(t,a)| dt < \infty.$$
\end{definition}

\subsection{An Auxiliary Result and some Proofs}
The following auxiliary result is very helpful for our convergence statements. For a proof see e.g. \cite{lange17}, Lemma B.12:

\begin{lemma}\label{le:helpA1}
Let $X$ be a separable and metrizable space, $Y$ a compact metric space and $f: X \times Y \to \R$ continuous. Then $\lim_{n\to\infty} x_n = x$ implies
$$ \lim_{n\to\infty} \sup_{y \in Y} \left|f(x_n,y) - f(x,y) \right| =0.$$
\end{lemma}

\vspace*{0.4cm}
\noindent
{\bf Proof of Lemma \ref{le:LambdaRIsContinuous}:}

Note that by definition $\Gamma^r(y,t) = \beta t + \Lambda^r(y,t)$. Thus, it is enough to show that the mapping $r \mapsto \Lambda^r(y,t)$ is continuous. Let $y \in E^0$ and $t \geq 0$. Further, let $(r^n)$ be a sequence in $\mathcal{R}$ with $\lim_{n\to\infty} r^n = r \in \mathcal{R}$. By definition of $\Lambda^r$, we then get:
\begin{eqnarray}
\nonumber && \left| \Lambda^{r^n}(y,t) - \Lambda^r(y,t)\right| \\
\nonumber &=& \left|\int_0^t \int_A \lambda^{A}\left(\Phi^{r^n}(y,s),a\right) ~r^n_s(da) ~ds- \int_0^t \int_A \lambda^{A}\left(\Phi^{r}(y,s),a\right) ~r_s(da) ~ds\right|\\
\nonumber &\leq & \left| \int_0^t \int_A \left\{ \lambda^{A}\left(\Phi^{r^n}(y,s),a\right) - \lambda^{A}\left(\Phi^{r}(y,s),a\right) \right\} r^n_s(da) ~ds \right|\\
\label{eq:StetigkeitLambda} && + \left|\int_0^t \int_A \lambda^{A}\left(\Phi^{r}(y,s),a\right) ~r^n_s(da) ~ds - \int_0^t \int_A \lambda^{A}\left(\Phi^{r}(y,s),a\right) ~r_s(da) ~ds \right|.
\end{eqnarray}
Looking now at the first summand of \eqref{eq:StetigkeitLambda} we find that for $n\to\infty$
\begin{eqnarray*}
&&\left| \int_0^t \int_A \left\{ \lambda^{A}\left(\Phi^{r^n}(y,s),a\right) - \lambda^{A}\left(\Phi^{r}(y,s),a\right) \right\} r^n_s(da) ~ds \right|\\
 &\leq & \int_0^t \sup_{a\in A} \left|\lambda^{A}(\Phi^{r^n}(y,s),a) - \lambda^{A}(\Phi^{r}(y,s),a) \right| ~ds \to 0.
\end{eqnarray*}
This convergence is true since  by the continuity of $\Phi^r$and $\lambda^A$ and by the compactness of $A$ we have with the help of Lemma \ref{le:helpA1}
$$\lim_{n\to\infty}   \sup_{a\in A} \left|\lambda^{A}(\Phi^{r^n}(y,s),a) - \lambda^{A}(\Phi^{r}(y,s),a) \right| = 0.$$
By the boundedness of $\lambda^{A}$, dominated convergence leads to the convergence of the integral towards zero.

Now, looking at the second summand in \eqref{eq:StetigkeitLambda} we obtain by the characterization of the Young topology (see Definition \ref{YOUNG-4})  and by assumption (C2) that
\begin{equation} \label{eq:IntegralConvergenceLambda} \lim_{n\to\infty} \int_0^t \int_A \lambda^{A}\left(\Phi^{r}(y,s),a\right) ~r^n_s(da) ~ds = \int_0^t \int_A \lambda^{A}\left(\Phi^{r}(y,s),a\right) ~r_s(da) ~ds.
\end{equation}
This implies the statement.

\vspace*{0.4cm}
\noindent
{\bf Proof of Lemma \ref{theo:GPrimeLowerSemiContinuous}:}

We first show that when $c$ is continuous and bounded, then $(\rho,r)\mapsto \hat{g}(\rho,r)$ is continuous and bounded. Let $c$ be continuous and bounded  and suppose  $\lim_{n\to\infty} (\rho^n, r^n) = (\rho,r)$ with respect to the product topology. Let us denote $\eta^r(y,t) := e^{-\Gamma^r(y,t)}$. Based on the representation of $g$ we then get
\begin{eqnarray*}
\left| \hat{g}(\rho^n, r^n) - \hat{g}(\rho,r)\right| &\leq& \sum_{y\in E^0} \left| \rho^n(y) \int_0^\infty \eta^{r^n}(y,t) \int_A c(\Phi^{r^n}(y,t),a) ~r^n_t(da) ~dt\right.\\
&& - \left. \rho(y) \int_0^\infty \eta^{r}(y,t) \int_A c(\Phi^{r}(y,t),a) ~r_t(da) ~dt \right|.
\end{eqnarray*}
From our assumption it follows that we obtain pointwise convergence $\lim_{n\to\infty} \rho^n(y) = \rho(y)$   and it thus remains to show that for all $y\in E^0$
$$\lim_{n\to\infty}\int_0^\infty \eta^{r^n}(y,t) \int_A c(\Phi^{r^n}(y,t),a) ~r^n_t(da) ~dt = \int_0^\infty \eta^{r}(y,t) \int_A c(\Phi^{r}(y,t),a) ~r_t(da) ~dt.$$
Hence consider
\begin{eqnarray}\label{eq:GPrimeContinuousAbsch}
\nonumber && \left| \int_0^\infty \eta^{r^n}(y,t) \int_A c(\Phi^{r^n}(y,t),a) ~r^n_t(da) ~dt - \int_0^\infty \eta^{r}(y,t) \int_A c(\Phi^{r}(y,t),a) ~r_t(da) ~dt \right|\\
\nonumber &\leq& \left| \int_0^\infty \left(\eta^{r_n}(y,t) - \eta^r(y,t)\right) \int_A c(\Phi^{r^n}(y,t),a) ~r^n_t(da) ~dt \right|\\
\label{eq:rdiff} && + \left| \int_0^\infty \eta^r(y,t)  \left\{ \int_A c(\Phi^{r^n}(y,t),a) ~r^n_t(da) - \int_A c(\Phi^{r}(y,t),a) ~r_t(da)\right\} ~dt\right|.
\end{eqnarray}
Now, as $c$ is bounded by our initial assumption, the first summand satisfies
\begin{eqnarray*}
&& \left| \int_0^\infty \left(\eta^{r^n}(y,t) - \eta^r(y,t)\right) \int_A c(\Phi^{r^n}(y,t),a) ~r^n_t(da) ~dt \right|\\
 &\leq& \sup_{x,a} |c(x,a)| \int_0^\infty \left| \eta^{r^n}(y,t) - \eta^r(y,t) \right| ~dt \to 0.
\end{eqnarray*}
The convergence follows from dominated convergence where $|\eta^{r^n}(y,t) - \eta^r(y,t)|$ is dominated by $2  e^{-\beta t}$ and $\lim_{n\to\infty} \Gamma^{r^n}(y,t) =  \Gamma^{r}(y,t)$ because of Lemma \ref{le:LambdaRIsContinuous}.

The second summand of \eqref{eq:rdiff} can be dominated by $Term_1+Term_2$ with
$$Term_1:=\int_0^\infty \eta^r(y,t)  \int_A \left| c(\Phi^{r^n}(y,t),a) -  c(\Phi^{r}(y,t),a) \right| ~r^n_t(da) ~dt$$
and
$$
Term_2:= \left| \int_0^\infty \eta^r(y,t)  \int_A c(\Phi^{r}(y,t),a) ~r^n_t(da) ~dt - \int_0^\infty \eta^r(y,t)  \int_A c(\Phi^{r}(y,t),a) ~r_t(da) ~dt \right|.
$$
We will show that both, $Term_1$ and $Term_2$ converge to zero. First,  as $c$ is continuous and bounded and  $A$ is compact we obtain with the help of Lemma \ref{le:helpA1}
$$\lim_{n\to\infty} \sup_{a\in A} \left| c(\Phi^{r^n}(y,t),a) -  c(\Phi^{r}(y,t),a) \right| = 0.$$
Thus $Term_1$ converges to zero by dominated convergence applied for dominating function $t\mapsto 2 \sup_{x,a} | c(x,a)|\eta^r(y,t)$.
For $Term_2$ we get convergence to zero from the characterization of the Young topology convergence in Definition \ref{YOUNG-4} as
$$(t,a) \mapsto \eta^r(y,t)  c(\Phi^r(y,t), a)$$
is measurable in $t$ and continuous and bounded  in $a$ and because of
$$\int_0^\infty  \eta^r(y,t)\; \sup_{a\in A}|c(\Phi^r(y,t), a)| ~dt \leq \sup_{(x,a)\in \R^d\times A}| c(x,a) | \int_0^\infty e^{-\beta t} ~dt < \infty.$$

We also get that $\hat{g}$ is bounded when $c$ is bounded.

Now, let $c$ be lower semi-continuous (and non-negative, what we always assume). Then, there is a sequence $(c_m)$ of continuous and bounded functions with $c_m \uparrow c$ for $m \to \infty$ (see \cite{BertShrev}, Lemma 7.14). Thus we can apply our previous findings to $c_m$ and by monotonicity of the convergence obtain that $\hat{g}$ is lower semi-continuous.

\vspace*{0.4cm}
\noindent
{\bf Proof of Lemma \ref{lem:Continuous2}:}
 Suppose $(r^n) \subset \mathcal{R}$ and $r^n \to r \in \mathcal{R}$ for $n \to \infty$.  Let us denote $\eta^r(y,s) := e^{-\Gamma^r(y,s)}$ and consider $\int \tilde{q}(s,y',x|y,r)ds$. Obviously the factor $f(x-y')$ does not depend on $r$ and can be ignored. We obtain
      \begin{eqnarray*}
   && \left| \int_0^\infty \eta^{r^n}(y,s) \int_A \lambda^A\left(\Phi^{r^n}(y,s),a\right) Q^{A}\left(y' | \Phi^{r^n}(y,s),a\right) r^n_s(da) ~ds-  \right. \\
  &&\hspace*{1cm} - \left. \int_0^\infty \eta^{r}(y,s) \int_A \lambda^A\left(\Phi^{r}(y,s),a\right) Q^{A}\left(y' | \Phi^{r}(y,s),a\right) r_s(da) ~ds \right|\\
   &\leq&  \left| \int_0^\infty  \int_A \Big\{ \eta^{r^n}(y,s) \lambda^A\left(\Phi^{r^n}(y,s),a\right) Q^{A}\left(y' |\Phi^{r^n}(y,s),a\right)- \right. \\
   && \hspace*{1cm} - \left. \eta^{r}(y,s) \lambda^A\left(\Phi^{r}(y,s),a\right) Q^{A}\left(y' |\Phi^{r}(y,s),a\right)\Big\}  r^n_s(da) ~ds \right|  \\
   &+& \left| \int_0^\infty   \int_A \eta^{r}(y,s)  \lambda^A\left(\Phi^{r}(y,s),a\right) Q^{A}\left(y' | \Phi^{r}(y,s),a\right) r^n_s(da)ds-  \right. \\
   &&\hspace*{1cm} - \left.  \int_0^\infty   \int_A \eta^{r}(y,s)  \lambda^A\left(\Phi^{r}(y,s),a\right) Q^{A}\left(y' | \Phi^{r}(y,s),a\right) r_s(da)ds\right|.
\end{eqnarray*}
 The first of these two terms converges to zero with Lemma \ref{le:helpA1}. The second term converges to zero by the definition of the Young topology and the fact that
 $$  \int_0^\infty \eta^r(y,s) \sup_{a\in A}| \lambda^A\left(\Phi^{r}(y,s),a\right) Q^{A}\left(y' | \Phi^{r}(y,s),a\right)| ds<\infty. $$

\vspace*{0.4cm}
\noindent
{\bf Proof of Lemma \ref{lem:Continuous}:} In the same way as in the proof of Lemma \ref{lem:Continuous2} it can be shown that
$$ r \mapsto \int_{\R_+} h_\sigma(s-u) \tilde{q}(u,y',x|y,r) du$$
is continuous. The statement follows since $(\rho,r)\mapsto \hat{\Psi}(\rho,r,x,u)$ is a continuous composition of these functions.

\vspace*{0.4cm}
\noindent
{\bf Proof of Lemma \ref{lem:qprimeweaklycont}:}
\begin{proof}
We have to show that
$$ (\rho,r) \mapsto \int_{\R_+} \int_{\R^d} \sum_y v(\hat{\Psi}(\rho,r,s,x)) \tilde{q}^{SX}(s,x|y,r) \nu(dx)ds \rho(y)$$
is continuous for $v$ bounded continuous. Obviously it is enough to show for fixed $y\in E^0$ that
$$ (\rho,r) \mapsto \int_{\R_+}\int_{\R^d}  v(\hat{\Psi}(\rho,r,s,x)) \tilde{q}^{SX}(s,x|y,r) \nu(dx)ds $$
is continuous. Let $\lim_{n\to\infty} (\rho^n,r^n)=(\rho,r)$ w.r.t.\ the product topology. We obtain:
\begin{eqnarray*}
&&\Big| \int\limits_{\R_+}\int\limits_{\R^d} v\big(\hat{\Psi}(\rho^n,r^n,s,x)\big)\tilde{q}^{SX}(s,x|y,r^n) \nu(dx)ds - \int\limits_{\R_+}\int\limits_{\R^d} v\big(\hat{\Psi}(\rho,r,s,x)\big)\tilde{q}^{SX}(s,x|y,r) \nu(dx)ds \Big| \\
&\le&  \Big| \int_{\R_+}\int_{\R^d} v\big(\hat{\Psi}(\rho^n,r^n,s,x)\big)\Big( \tilde{q}^{SX}(s,x|y,r^n) -\tilde{q}^{SX}(s,x|y,r)\Big) \nu(dx)ds \Big| \\
&&  \hspace*{2cm}+ \Big| \int_{\R_+}\int_{\R^d}\Big( v\big(\hat{\Psi}(\rho^n,r^n,s,x) \big)- v\big(\hat{\Psi}(\rho,r,s,x)\Big)  \tilde{q}^{SX}(s,x|y,r) \nu(dx)ds \Big|.
\end{eqnarray*}
Since $v$ is bounded by a constant, the first term can be bounded by
\begin{equation*}\sup_\rho |v(\rho)| \times \Big| \int_{\R_+}\int_{\R^d}  \tilde{q}^{SX}(s,x|y,r^n) -\tilde{q}^{SX}(s,x|y,r)\nu(dx)ds \Big|
\end{equation*}
which converges to zero for $n\to\infty$ because of Lemma \ref{lem:Continuous2} and dominated convergence. The second term converges to zero by dominated convergence and continuity of $v$ and $\hat{\Psi}$.
\end{proof}

\bibliographystyle{abbrv}

\end{document}